\documentclass[12pt]{article}
\usepackage{amsmath,amsthm,amssymb}
\usepackage{graphicx}

\newtheorem{theorem}{Theorem}[section]
\newtheorem{lemma}{Lemma}[section]
\newtheorem{prop}{Proposition}[section]

\newtheorem{definition}{Definition}[section]
\newtheorem{rem}{Remark}[section]
\newtheorem{ex}{Example}[section]

\begin{document}
\thispagestyle{empty}
\begin{center}

{\bf\Large Exact stability and instability regions for two-dimensional linear autonomous multi-order systems of fractional-order differential equations}

\vspace{8mm}


{\large Oana Brandibur,\  Eva Kaslik}

\vspace{3mm}

{\em Department of Mathematics and Computer Science\\ West University of Timi\c soara, Romania \\
E-mail: oana.brandibur@e-uvt.ro, \ eva.kaslik@e-uvt.ro}

\vspace{5mm}
\end{center}
\vspace{8mm}

\fbox{
\begin{minipage}{33em}
This paper is now published (in revised form) in Fract. Calc. Appl. Anal. Vol. 24, No 1 (2021), pp. 225–253,  DOI: 10.1515/fca-2021-0010
\end{minipage}}

\vspace{8mm}

\textbf{Abstract:} Necessary and sufficient conditions are explored for the asymptotic stability and instability of linear two-dimensional autonomous systems of  fractional-order differential equations with Caputo derivatives. 
Fractional-order-dependent and fractional-order-independent stability and instability properties are fully characterised, in terms of the main diagonal elements of the systems' matrix, as well as its determinant.

 \medskip

\section{Introduction}
\label{sec:1}
Within the past decades, a growing number of scientific papers debated the pertinence of fractional calculus in the mathematical modeling of real world phenomena, suggesting that fractional-order systems are capable of delivering more realistic results in a large number of practical applications \cite{Cottone,Engheia,Henry_Wearne,Heymans_Bauwens,Mainardi_1996} compared to their integer-order counterparts. The main justification of this fact is that fractional-order derivatives provide for the incorporation of both memory and hereditary properties. Indeed, \cite{du2013measuring} endorses the index of memory as a plausible physical interpretation of the order of a fractional derivative. 

As in the case of classical dynamical systems theory, stability analysis plays a leading role in the qualitative theory of fractional-order systems. Two surveys \cite{Li-survey,Rivero2013stability} have recently summarized the main results that have been obtained with respect to the stability properties of fractional-order systems. Nevertheless, it has to be emphasized that most results have been obtained in the framework of linear autonomous commensurate fractional-order systems. In this context, it is important to note that a generalization of the well-known stability theorem of Matignon \cite{Matignon} has been recently obtained \cite{Sabatier2012stability}. Furthermore, linearization theorems for fractional-order systems have been presented in \cite{Li_Ma_2013,Wang2016stability}, providing analogues of the classical Hartman-Grobman theorem. 

On the other hand, the stability analysis of incommensurate fractional-order systems has received significantly less attention throughout the years. Stability properties of linear incommensurate fractional-order systems with rational orders have been investigated in \cite{Petras2008stability}. Oscillatory behaviour in two-dimensional incommensurate fractional-order systems has been explored in \cite{Datsko2012complex,Radwan2008fractional}. Bounded input bounded output stability of systems with irrational transfer functions has been recently analyzed in \cite{trachtler2016bibo}. The asymptotic behavior of the solutions of some classes of linear multi-order systems of fractional differential equations (such as systems with block triangular coefficient matrices) has been investigated in \cite{diethelm2017asymptotic}.

Multi-term fractional-order differential equations \cite{atanackovic2014cauchy} and their stability properties are closely related to multi-order systems of fractional differential equations. Very recently, the stability of two-term fractional-order differential and difference equations has been analyzed in   \cite{cermak2015asymptotic,cermak2015stability,jiao2012stability}.

Taking into account the above mentioned  developments in the theory of fractional-order systems, necessary and sufficient stability and instability conditions have been explored in the case of linear autonomous two-dimensional incommensurate fractional-order systems \cite{Brandibur_2017, Brandibur_2018}. In the first paper \cite{Brandibur_2017}, we have investigated stability properties of two-dimensional systems composed of a fractional-order differential equation and a classical first-order differential equation. These results have been extended in \cite{Brandibur_2018} for the case of general two-dimensional incommensurate fractional-order systems with Caputo derivatives. Specifically, for fractional orders $0<q_1<q_2\leq 1$, necessary and sufficient conditions have been obtained for the $\mathcal{O}(t^{-q_1})$-asymptotic stability of the trivial equilibrium, in terms of the determinant $\delta$ of the linear system's matrix, as well as the elements $a_{11}$ and $a_{22}$ of its main diagonal. Moreover, sufficient conditions have also been investigated which guarantee the stability and instability of the fractional-order system, regardless of the fractional orders. 

The aim of this work is to complete the stability analysis of two-dimensional incommensurate fractional-order systems with Caputo derivatives, by extending the results presented in \cite{Brandibur_2018,brandibur2019_gejji}. On one hand, we fully characterize the fractional-order dependent stability and instability properties of the considered system, by exploring certain symmetries related to the characteristic equation associated to our stability problem. On the other hand, we obtain necessary and sufficient conditions for the stability and instability of the system, regardless of the choice of fractional orders, in terms of the characteristic parameters $a_{11}$, $a_{22}$ and $\delta$ mentioned previously. These latter results are particularly useful in practical applications where the exact fractional orders are not precisely known. 

The paper is structured as follows. Section 2 is dedicated to presenting some preliminary results and important definitions. The main results are included in section 3 as follows: we first present the statements of the main fractional-order-independent stability and instability theorems, then we prove fractional-order-dependent stability and instability results, followed by the proofs of the main theorems. For the sake of completeness, all proofs are presented in detail. Finally, we draw some conclusions and suggest several directions for future research in section 4.  

\section{Preliminaries}
\label{sec:2}

Let us consider the  $n$-dimensional fractional-order system
with Caputo derivatives \cite{Kilbas,Lak,Podlubny}: \begin{equation}\label{sys.gen}
^c\!D^\mathbf{q}\mathbf{x}(t)=f(t,\mathbf{x})
\end{equation}
where $\mathbf{q}=(q_1,q_2,...,q_n)\in(0,1)^n$ and $f:[0,\infty)\times\mathbb{R}^n\rightarrow \mathbb{R}^n$ is a continuous function on the whole domain of definition, Lipschitz-continuous with respect to the second variable, such that
$$f(t,0)=0\quad \textrm{for any }t\geq 0.$$

Let $\varphi(t,\mathbf{x}_0)$ denote the unique solution of (\ref{sys.gen}) satisfying the initial condition $\mathbf{x}(0)=\mathbf{x}_0\in\mathbb{R}^n$. The existence and uniqueness of the initial value problem associated to system (\ref{sys.gen}) is guaranteed by the previously mentioned properties of the function $f$ \cite{Diethelm_book}.

It is important to emphasize that in general, due to the presence of the memory effect, the asymptotic stability of the trivial solution of system (\ref{sys.gen}) is not of exponential type \cite{cermak2015stability,Gorenflo_Mainardi}. Hence, the notion of Mittag-Leffler stability has been introduced for fractional-order differential equations \cite{Li_Chen_Podlubny}, as a special type of non-exponential asymptotic stability concept. In this work, we focus on  $\mathcal{O}(t^{-\alpha})$-asymptotic stability, reflecting the algebraic decay of the solutions.

\begin{definition}\label{def.stability}
	The trivial solution of (\ref{sys.gen}) is called \emph{stable} if for any $\varepsilon>0$
	there exists $\delta=\delta(\varepsilon)>0$ such that for every $\mathbf{x}_0\in\mathbb{R}^n$ satisfying $\|\mathbf{x}_0\|<\delta$ we have
	$\|\varphi(t,\mathbf{x}_0)\|\leq\varepsilon$ for any $t\geq 0$.
	
	The trivial solution  of (\ref{sys.gen}) is called \emph{asymptotically stable} if it is stable and t here
	exists $\rho>0$ such that $\lim\limits_{t\rightarrow\infty}\varphi(t,\mathbf{x}_0)=0$ whenever $\|\mathbf{x}_0\|<\rho$.
	
	Let $\alpha>0$. The trivial solution  of (\ref{sys.gen}) is called \emph{$\mathcal{O}(t^{-\alpha})$-asymptotically stable} if it is stable and there exists $\rho>0$ such that for any $\|\mathbf{x}_0\|<\rho$ one has:
	$$\|\varphi(t,\mathbf{x}_0)\|=\mathcal{O}(t^{-\alpha})\quad\textrm{as }t\rightarrow\infty.$$
\end{definition}

\section{Main results}
In this paper, we consider the following two-dimensional linear autonomous incommensurate fractional-order system:
\begin{equation}\label{linearsys}
\left\{
\begin{array}{l}
^c\!D^{q_1}x(t)=a_{11}x(t)+a_{12}y(t) \\
^c\!D^{q_2}y(t)=a_{21}x(t)+a_{22}y(t)
\end{array}
\right.
\end{equation}
where $A=(a_{ij})$ is a real two-dimensional matrix and $q_1,q_2\in(0,1]$ are the fractional orders of the Caputo derivatives. The following characteristic equation is obtained by means of the Laplace transform method:
$$
\det\left(\text{diag}(s^{q_1},s^{q_2})-A\right)=0
$$
which is equivalent to
\begin{equation}\label{eq.char}
s^{q_1+q_2}-a_{11}s^{q_2}-a_{22}s^{q_1}+\det(A)=0.
\end{equation}
It is important to emphasize that in the characteristic equation \eqref{eq.char}, $s^{q_1}$ and $s^{q_2}$ represent the principal values (first branches) of the corresponding complex power functions \cite{Doetsch}.

By means of asymptotic expansion properties and the Final Value Theorem of the Laplace transform \cite{Bonnet_2000,Brandibur_2017,Doetsch}, necessary and sufficient conditions for the global asymptotic stability of system (\ref{linearsys}) have been recently obtained \cite{Brandibur_2018}:
\begin{prop}\label{thm.lin.stab}$ $
	\begin{enumerate}
		\item Denoting $q=\min\{q_1,q_2\}$, system (\ref{linearsys}) is $\mathcal{O}(t^{-q})$-globally asymptotically stable if and only if all the roots of the characteristic equation \eqref{eq.char} are in the open left half-plane.
		\item If $\det(A)\neq 0$ and the characteristic equation \eqref{eq.char} has a root in the open right half-plane, system (\ref{linearsys}) is unstable.
	\end{enumerate}
\end{prop}

The aim of this paper is to analyze the distribution of the roots of the characteristic equation \eqref{eq.char} with respect to the imaginary axis of the complex plane. In what follows, we denote $\det(A)=\delta$ and we consider the complex-valued function 
$$\Delta(s;a_{11},a_{22},\delta,q_1,q_2)=s^{q_1+q_2}-a_{11}s^{q_2}-a_{22}s^{q_1}+\delta$$
which gives the left-hand side of the characteristic equation \eqref{eq.char}. 

\begin{rem}
The analysis of the roots of the characteristic function $\Delta(s;a_{11},a_{22},\delta,q_1,q_2)$ is also encountered in the investigation of the stability properties of the three-term fractional-order differential equation
\begin{equation}\label{eq.multi.term}
    ^c\!D^{q_1+q_2}x(t)-a_{11}~\! ^c\!D^{q_2}x(t)-a_{22}~\! ^c\!D^{q_1}x(t)+\delta x(t)=0.
\end{equation}
Therefore, the results presented in this paper are also applicable in the framework of equation \eqref{eq.multi.term}.
\end{rem}

The statements of the main results are presented below, followed by detailed proofs in the upcoming sections.

\subsection{Fractional-order-independent stability and instability results}
\label{sec:indep}$ $

Obtaining fractional-order-independent necessary and sufficient conditions for the asymptotic stability or instability of system (\ref{linearsys})  are particularly useful in practical applications where the exact values of the fractional orders used in the mathematical modeling are not  precisely known. In this section, we only state the main results, giving their complete proofs in section \ref{sec:proofs}, due to their complexity.  

\begin{theorem}[Fractional-order independent instability results]\label{thm.instab}
$ $
\begin{itemize}
\item[i.] If $\det(A)<0$, system (\ref{linearsys}) is unstable, regardless of the fractional orders $q_1$ and $q_2$.

\item[ii.] If $\det(A)>0$, system (\ref{linearsys}) is unstable regardless of the fractional orders $q_1$ and $q_2$ if and only if one of the following conditions holds:
$$ 
\begin{cases}
& a_{11}+a_{22}\geq \det(A)+1 \ \textrm{or} \\
&  a_{11}>0,~a_{22}>0,~ a_{11}a_{22}\geq \det(A).
\end{cases}
$$
\end{itemize}
\end{theorem}

\begin{theorem}[Fractional-order-independent stability results]\label{thm.stab}
System (\ref{linearsys}) is asymptotically stable, regardless of the fractional orders $q_1,q_2\in(0,1]$ if and only if the following inequalities are satisfied:
$$
a_{11}+a_{22}<0<\det(A)\quad \text{and}\quad \max\{a_{11},a_{22}\}<\min\{1,\det(A)\}.
$$
\end{theorem}

\begin{rem}
In the classical integer order case (i.e. $q_1=q_2=1$), it is well-known that a two-dimensional linear autonomous system of the form $\textbf{x}'=A\mathbf{x}$, with constant matrix $A\in\mathbb{R}^{2\times 2}$ is asymptotically stable if and only if $\textrm{Tr}(A)<0$ and $\det(A)>0$. Based on Theorem \ref{thm.stab}, a supplementary inequality $$\max\{a_{11},a_{22}\}<\min\{1,\det(A)\}\}$$ is required to guarantee that system \eqref{linearsys} is asymptotically stable, regardless of the choice of fractional orders $q_1,q_2\in(0,1]$.
\end{rem}

Based on the previous theorems, as the case $\det(A)<0$ is trivial (i.e. system \eqref{linearsys} is unstable for any $q_1,q_2\in(0,1]$), in what follows, we consider $\det(A)=\delta>0$ and we define the following regions in the $(a_{11},a_{22})$-plane:
\begin{align*}
 R_u(\delta)&\!\!=\!\!\{(a_{11},a_{22})\in\mathbb{R}^2: a_{11}+a_{22}\geq \delta+1 \ \textrm{or}\  a_{11}>0,~a_{22}>0,~ a_{11}a_{22}\geq \delta\}\\
 R_s(\delta)&\!\!=\!\!\{(a_{11},a_{22})\in\mathbb{R}^2: a_{11}+a_{22}<0\ \text{and}\ \max\{a_{11},a_{22}\}<\min\{1,\delta\}\}
\end{align*} 
An example is presented for the particular case $\delta=4$ in Figure \ref{fig.regiuni}.

\begin{figure}[htbp]
\centering
\includegraphics*[width=0.46\linewidth]{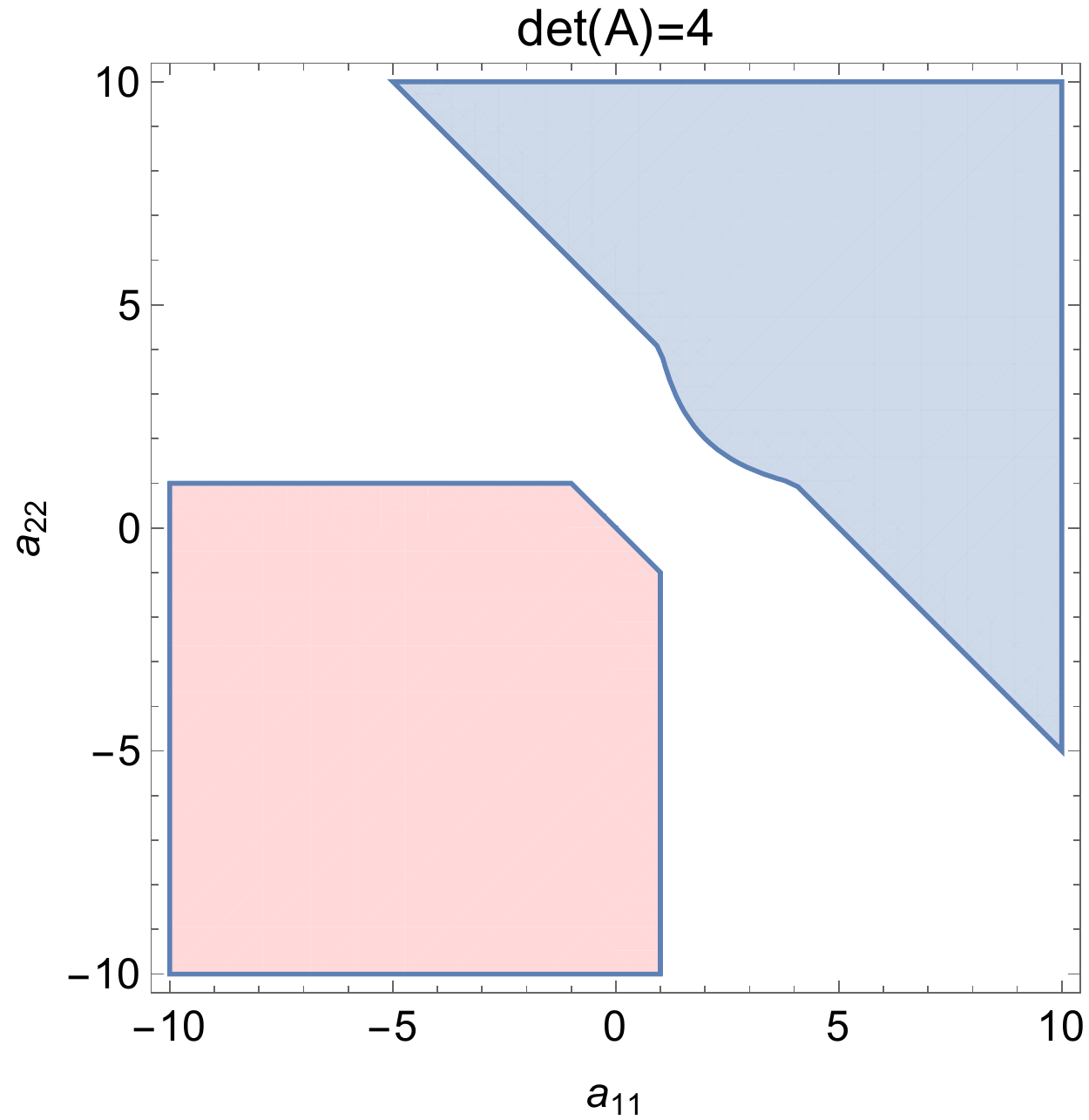}
\caption{The red/blue shaded regions represent the sets $R_s(\delta)$ and $R_u(\delta)$, respectively, for $\delta=\det(A)=4$.}
\label{fig.regiuni}
\end{figure}

\begin{rem}
Due to Theorem \ref{thm.instab}, when $\det(A)=\delta>0$ is arbitrarily fixed, system \eqref{linearsys} is unstable for any choice of the fractional orders $q_1,q_2\in(0,1]$ if and only if $(a_{11},a_{22})\in R_u(\delta)$. On the other hand, based on Theorem \ref{thm.stab}, system \eqref{linearsys} is asymptotically stable for any $q_1,q_2\in(0,1]$ if and only if $(a_{11},a_{22})\in R_s(\delta)$. Therefore, if $(a_{11},a_{22})\in\mathbb{R}^2\setminus(R_s(\delta)\cup R_u(\delta))$ (e.g. white region in Fig. \ref{fig.regiuni}), the stability properties of system \eqref{linearsys} depend on the considered fractional orders.     
\end{rem}

\subsection{Fractional-order-dependent stability and instability results}
$ $

The aim of this section is to characterize the stability properties of system
\eqref{linearsys} when $\det(A)=\delta>0$ and $q_1,q_2\in(0,1]$ are arbitrarily fixed. The case $\delta<0$ is not considered here, as from Theorem \ref{thm.instab} we know that in this case, system \eqref{linearsys} is unstable for any $q_1,q_2\in(0,1]$.

\begin{lemma}\label{lem.curve.gamma}
	Let $\delta>0$, $q_1,q_2\in(0,1]$ and consider the smooth parametric curve in the $(a_{11},a_{22})$-plane defined by
	$$
	\Gamma(\delta,q_1,q_2)~:\quad
	\begin{cases}
		a_{11}=\delta^{\frac{q_1}{q_1+q_2}}~h(\omega,q_1,q_2)\\
        a_{22}=\delta^{\frac{q_2}{q_1+q_2}}~h(-\omega,q_1,q_2)
	\end{cases},\quad \omega\in\mathbb{R},
	$$
	where:
    $$h(\omega,q_1,q_2)=
    \begin{cases}
    \rho_2(q_1,q_2)e^{q_1\omega}-\rho_1(q_1,q_2)e^{-q_2\omega}, &\text{if }q_1\neq q_2\\
    \cos\frac{q\pi}{2}-\omega, &\text{if }q_1=q_2:=q
    \end{cases}$$
with the functions $\rho_1(q_1,q_2)$ and $\rho_2(q_1,q_2)$ defined for $q_1\neq q_2$ as
$$\rho_k(q_1,q_2)=\frac{\sin\frac{q_k\pi}{2}}{\sin\frac{(q_2-q_1)\pi}{2}}\quad,\quad\text{for }k=\overline{1,2}.$$
The following statements hold:
\begin{itemize}	
\item[i.] The curve $\Gamma(\delta,q_1,q_2)$ is the graph of a smooth, decreasing, concave bijective function $\phi_{\delta,q_1,q_2}:\mathbb{R}\rightarrow\mathbb{R}$ in the $(a_{11},a_{22})$-plane.
\item[ii.] The curve $\Gamma(\delta,q_1,q_2)$ lies outside the third quadrant of the $(a_{11},a_{22})$-plane.
\end{itemize}
\end{lemma}

\begin{proof} Let $\delta>0$ and $q_1,q_2\in(0,1]$ arbitrarily fixed. 

\emph{Proof of statement (i).}  The real-valued function $\omega\mapsto h(\omega,q_1,q_2)$ is bijective and monotonous on $\mathbb{R}$: strictly decreasing if $q_1\leq q_2$ and strictly increasing otherwise. Therefore, the particular form of the parametric equations  implies that the curve $\Gamma(\delta,q_1,q_2)$ is the graph of a smooth decreasing bijective function $\phi_{\delta,q_1,q_2}:\mathbb{R}\rightarrow\mathbb{R}$ in the $(a_{11},a_{22})$-plane. 

If $q_1\neq q_2$, using the chain rule, we compute: 
$$\dfrac{d^2a_{22}}{da_{11}^2}=\delta^{\frac{q_2-2q_1}{q_1+q_2}}\cdot\dfrac{\rho_1\rho_2q_1 q_2(q_1-q_2)\left[e^{(q_1+q_2)\omega}+e^{-(q_1+q_2)\omega}\right]+2(q_1^3 \rho_2^2-q_2^3 \rho_1^2)}{\left(q_1\rho_2e^{q_1\omega}+q_2\rho_1 e^{-q_2\omega}\right)^3}$$
Assuming that $q_1<q_2$, the expression above is strictly negative, as $\rho_1>0$, $\rho_2>0$ and $\frac{q_1}{q_2}\leq \frac{\rho_1}{\rho_2}\leq 1$ (since the function $x\mapsto\frac{\sin x}{x}$ is decreasing on $(0,\pi)$). A similar argument holds in the case $q_1>q_2$ as well. Hence, $\phi_{\delta,q_1,q_2}$ is a concave function. 

\emph{Proof of statement (ii).} Assume the contrary, i.e. that there exists $(a_{11},a_{22})\in\Gamma(\delta,q_1,q_2)$ such that $a_{11}<0$ and $a_{22}<0$, or equivalently, that there exists $\omega\in\mathbb{R}$ such that $h(\pm\omega,q_1,q_2)<0$. As the case $q_1=q_2$ is trivial, we assume without loss of generality that $q_1<q_2$. The inequalities $h(\pm\omega,q_1,q_2)<0$ are equivalent to
$$\rho_2(q_1,q_2)e^{\pm(q_1+q_2)\omega}<\rho_1(q_1,q_2)$$
which leads to $\rho_2(q_1,q_2)<\rho_1(q_1,q_2)$, or equivalently to $q_2<q_1$, which is absurd. Hence, the curve $\Gamma(\delta,q_1,q_2)$ does not have any points in the third quadrant.
\end{proof}

\begin{rem}\label{rem.q1.q2.q}
If $q_1=q_2:=q$, $\Gamma(\delta,q_1,q_2)$ represents the straight line: $$a_{11}+a_{22}=2\sqrt{\delta}\cos\dfrac{q\pi}{2}.$$
\end{rem}

In the following, we will denote by $N(a_{11},a_{22},\delta,q_1,q_2)$ the number of unstable roots ($\Re(s)\geq 0$) of the characteristic function $\Delta(s;a_{11},a_{22},\delta,q_1,q_2)$, including their multiplicities. The following lemma shows that the function $N(a_{11},a_{22},\delta,q_1,q_2)$ is well-defined and establishes important properties which will be useful in the proof of the main results. 

\begin{lemma}\label{lemma.number.roots} Let $\delta>0$, $q_1,q_2\in(0,1]$ be arbitrarily fixed. The following statements hold:
\begin{itemize}
\item[i.] The characteristic function $\Delta(s;a_{11},a_{22},\delta,q_1,q_2)$ has at most a finite number of roots satisfying $\Re(s)\geq 0$. 

\item[ii.]  The function $(a_{11},a_{22})\mapsto N(a_{11},a_{22},\delta,q_1,q_2)$ is continuous at all points $(a_{11},a_{22})$ that do not belong to the curve $\Gamma(\delta,q_1,q_2)$. Consequently, $N(a_{11},a_{22},\delta,q_1,q_2)$ is constant on each connected component of $\mathbb{R}^2\setminus \Gamma(\delta,q_1,q_2)$.
\end{itemize}
\end{lemma}

\begin{proof}
The first step of the proof (see Appendix A.1.) consists of showing that there exist a strictly decreasing function $l_{\delta,q_1,q_2}:\mathbb{R}^+\rightarrow\mathbb{R}^+$ and a strictly increasing function $L_{\delta,q_1,q_2}:\mathbb{R}^+\rightarrow\mathbb{R}^+$ such that any unstable root of $\Delta(s;a_{11},a_{22},\delta,q_1,q_2)$ is bounded by
\begin{equation}\label{ineq.s}
    l_{\delta,q_1,q_2}(\|a\|_p)\leq |s|\leq L_{\delta,q_1,q_2}(\|a\|_p)
\end{equation}
where $p=\frac{q_1+q_2}{2\min\{q_1,q_2\}}\geq 1$ and $a=(a_{11},a_{22})$. Moreover, $\|\cdot\|_p$ denotes the $p$-norm in $\mathbb{R}^2$.

\emph{Proof of statement (i).} Assuming the contrary, that there exists an infinite number of unstable roots, the Bolzano-Weierstrass theorem implies that there exists a convergent sequence of unstable roots $(s_j)$ with the limit $s_0\neq 0$ (since $\delta>0$), such that $\Re(s_0)\geq 0$. As the function $\Delta(s;a_{11},a_{22},\delta,q_1,q_2)$ is analytic in $\mathbb{C}\setminus\mathbb{R}^-$, by the principle of permanence it follows that it  is identically zero, which is absurd. Therefore, we obtain that $N(a_{11},a_{22},\delta,q_1,q_2)$ is finite.

\emph{Proof of statement (ii).}
Let $a^0=(a^0_{11},a^0_{22})\in\mathbb{R}^2\setminus \Gamma(\delta,q_1,q_2)$ and consider $r>0$ such that the open neighborhood $B_r(a^0)=\{a=(a_{11},a_{22})\in\mathbb{R}^2: \|a-a^0\|_p<r\}$ of the point $a^0$ is included in $\mathbb{R}^2\setminus \Gamma(\delta,q_1,q_2)$. 

For any $a=(a_{11},a_{22})\in B_r(a^0)$, we have:
$$\|a\|_p\leq \|a-a^0\|_p+\|a^0\|_p<r+\|a^0\|_p$$
and hence, inequality \eqref{ineq.s} implies that any root $s$ of $\Delta(s;a_{11},a_{22},\delta,q_1,q_2)$ such that $\Re(s)\geq 0$ satisfies:
$$ l_{\delta,q_1,q_2}(r+\|a^0\|_p)< |s|< L_{\delta,q_1,q_2}(r+\|a^0\|_p).$$
Denoting $m=l_{\delta,q_1,q_2}(r+\|a^0\|_p)$ and  $M=L_{\delta,q_1,q_2}(r+\|a^0\|_p)$, let us consider in the complex plane the simple closed curve $(\gamma)$, oriented counterclockwise, bounding the open set
$$D=\{s\in\mathbb{C}:\Re(s)> 0,m<|s|< M\}.$$
The above construction shows that for any $a=(a_{11},a_{22})\in B_r(a^0)$ all unstable roots of $\Delta(s;a_{11},a_{22},\delta,q_1,q_2)$ are inside the open set $D$. 

As $\Delta(s;a^0_{11},a^0_{22},\delta,q_1,q_2)\neq 0$ for any $s\in(\gamma)$, it is easy to see that $$d_0=\min\limits_{s\in(\gamma)}|\Delta(s;a^0_{11},a^0_{22},\delta,q_1,q_2)|>0.$$ 
Moreover, we consider $q\geq 1$ such that $\frac{1}{p}+\frac{1}{q}=1$ and denote:  $$r'=\min\left \{r,\frac{d_0}{\|(M^{q_2},M^{q_1})\|_q}\right\}.$$ 
Based on H\"{o}lder's inequality, it follows that for any $s\in (\gamma)$ and for any $a\in B_{r'}(a^0)\subset B_r(a^0)$, we have:
\begin{align*}
|\Delta & (s;a_{11},a_{2},\delta,q_1,q_2)-\Delta(s;a^0_{11},a^0_{22},\delta,q_1,q_2)|=\\
&=|(a_{11}-a^0_{11})s^{q_2}+(a_{22}-a^0_{22})s^{q_1}|\leq\\ &\leq|a_{11}-a^0_{11}|M^{q_2}+|a_{22}-a^0_{22}|M^{q_1}\leq\\
&\leq \|a-a^0\|_p\|(M^{q_2},M^{q_1})\|_q<\\
&<r'\|(M^{q_2},M^{q_1})\|_q\leq d_0\leq |\Delta(s;a^0_{11},a^0_{22},\delta,q_1,q_2)|.
\end{align*}

Rouch\'{e}'s theorem implies  $\Delta(s;a_{11},a_{22},\delta,q_1,q_2)$ and $\Delta(s;a^0_{11},a^0_{22},\delta,q_1,q_2)$ have the same number of roots in the domain $D$, 
and hence $$N(a_{11},a_{22},\delta,q_1,q_2)=N(a^0_{11},a^0_{22},\delta,q_1,q_2)\quad \text{for any}~~ a\in B_{r'}(a^0).$$ Therefore, the function
$(a_{11},a_{22})\mapsto N(a_{11},a_{22},\delta,q_1,q_2)$ is continuous on $\mathbb{R}^2\setminus\Gamma(\delta,q_1,q_2)$, and as it is integer-valued, it follows that it is constant on each connected component of $\mathbb{R}^2\setminus\Gamma(\delta,q_1,q_2)$.
\end{proof}

The following theorem represents the main result characterizing  fractional-order-dependent stability and instability properties of system (\ref{linearsys}). 

\begin{theorem}[Fractional-order-dependent stability and instability results]\label{thm.q}
Let $\det(A)=\delta>0$ and $q_1,q_2\in(0,1]$ arbitrarily fixed. Consider the curve $\Gamma(\delta,q_1,q_2)$ and the function $\phi_{\delta,q_1,q_2}:\mathbb{R}\rightarrow\mathbb{R}$ given by Lemma \ref{lem.curve.gamma}.
\begin{itemize}
\item[i.] The characteristic equation \eqref{eq.char} has a pair of pure imaginary roots if and only if $(a_{11},a_{22})\in \Gamma(\delta,q_1,q_2)$.
			
\item[ii.] System (\ref{linearsys}) is $\mathcal{O}(t^{-q})$-asymptotically stable (with $q=\min\{q_1,q_2\}$) if and only if $$a_{22}<\phi_{\delta,q_1,q_2}(a_{11}).$$

\item[iii.] If $a_{22}>\phi_{\delta,q_1,q_2}(a_{11})$, system (\ref{linearsys}) is unstable.
\end{itemize}
\end{theorem}

\begin{proof} Assume that $\delta>0$ and $q_1,q_2\in(0,1]$ are arbitrarily fixed.

\emph{Proof of statement (i).} 
It is easy to see that the characteristic equation \eqref{eq.char} has a pair of pure imaginary roots if and only if there exists $\omega\in\mathbb{R}$ such that $\Delta(i\delta^{\frac{1}{q_1+q_2}}e^\omega; a_{11},a_{22},\delta,q_1,q_2)=0$. As $i^q=\cos \frac{q\pi}{2}+i\sin\frac{q\pi}{2}$, taking the real and the imaginary parts of the previous equation, one obtains:
\begin{equation}\label{sistemrealimaginar}
\!\!\begin{cases}
	a_{11}\delta^{-\frac{q_1}{q_1+q_2}}e^{q_2\omega}\cos\frac{q_2\pi}{2}+a_{22}\delta^{-\frac{q_2}{q_1+q_2}}e^{q_1\omega}\cos \frac{q_1\pi}{2}=e^{(q_1+q_2)\omega}\cos\frac{(q_1+q_2)\pi}{2}+1\\	
a_{11}\delta^{-\frac{q_1}{q_1+q_2}}e^{q_2\omega}\sin\frac{q_2\pi}{2}+a_{22}\delta^{-\frac{q_2}{q_1+q_2}}e^{q_1\omega}\sin \frac{q_1\pi}{2}=e^{(q_1+q_2)\omega}\sin\frac{(q_1+q_2)\pi}{2}
	\end{cases}
	\end{equation}

If $q_1\neq q_2$, solving this system for $a_{11}$ and $a_{22}$ shows that the characteristic equation \eqref{eq.char} has a pair of pure imaginary roots if and only if $(a_{11},a_{22})$ belongs to the curve $\Gamma(\delta,q_1,q_2)$ given by Lemma \ref{lem.curve.gamma}.

In the particular case $q_1=q_2:=q$, system \eqref{sistemrealimaginar} is compatible if and only if $\omega=0$. Moreover, the set of solutions of \eqref{sistemrealimaginar} is the straight line 
$$a_{11}+a_{22}=2\sqrt{\delta}\cos\frac{q\pi}{2}$$
which represents $\Gamma(\delta,q,q)$ given by Lemma \ref{lem.curve.gamma} (see Remark \ref{rem.q1.q2.q}).

\emph{Proof of statement (ii).} 
Choosing $a_{11}=a_{22}=-1$, we argue that $\Delta(s;-1,-1,\delta,q_1,q_2)$ does not have any roots in the right half plane. Indeed, assuming that there exists $s\in\mathbb{C}$ such that $\Re(s)\geq 0$ and 
$$s^{q_1+q_2}+s^{q_2}+s^{q_1}+\delta=0,$$
it follows by division by $s^{q_1}$ that 
$$s^{q_2}+s^{q_2-q_1}+1+\delta s^{-q_1}=0.$$
As $q_2\in(0,1]$, $q_2-q_1\in[-1,1]$ and $-q_1\in[-1,0)$, it follows that the real part of each term from the left hand side of the above equality is positive, which leads to a contradiction. Hence, $N(-1,-1,\delta,q_1,q_2)=0$. From Lemma \ref{lem.curve.gamma} (ii) and Lemma \ref{lemma.number.roots} it follows that $N(a_{11},a_{22},\delta,q_1,q_2)=0$, for any $(a_{11},a_{22})$ from the region below the curve $\Gamma(\delta,q_1,q_2)$, which leads to the desired conclusion.


\emph{Proof of statement (iii).}
Let $s(a_{11},a_{22},\delta,q_1,q_2)$ denote the root of $\Delta(s;a_{11},a_{22},\delta,q_1,q_2)$ satisfying $s(a^\star_{11},a^\star_{22},\delta,q_1,q_2)=i\beta$, with $\beta=\delta^{\frac{1}{q_1+q_2}}e^\omega$ as in the proof of statement (i), where $(a^\star_{11},a^\star_{22})\in\Gamma(\delta,q_1,q_2)$. Taking the derivative with respect to $a_{11}$ in the equation	$$s^{q_1+q_2}-a_{11}s^{q_2}-a_{22}s^{q_1}+\delta=0$$
we obtain
	$$(q_1+q_2)s^{q_1+q_2-1}\frac{\partial s}{\partial a_{11}}-s^{q_2}-a_{11}q_2s^{q_2-1}\frac{\partial s}{\partial a_{11}}-a_{22}q_1s^{q_1-1}\frac{\partial s}{\partial a_{11}}=0.$$
We deduce:
	$$\frac{\partial s}{\partial a_{11}}=\frac{s^{q_2}}{(q_1+q_2)s^{q_1+q_2-1}-a_{11}q_2s^{q_2-1}-a_{22}q_1s^{q_1-1}}$$
and therefore
	$$\frac{\partial \Re(s)}{\partial a_{11}}=\Re \left(\frac{\partial s}{\partial a_{11}}\right)=\Re\left( \frac{s^{q_2}}{(q_1+q_2)s^{q_1+q_2-1}-a_{11}q_2s^{q_2-1}-a_{22}q_1s^{q_1-1}} \right).$$
We have
$$
	\frac{\partial \Re(s)}{\partial a_{11}}\Big|_{(a^\star_{11},a^\star_{22})}=\Re \left( \frac{\left(i\beta\right)^{q_2}}{P(i\beta)} \right)=\beta^{q_2}\Re \left( \frac{i^{q_2}\overline{P(i\beta)}}{|P(i\beta)|^2} \right)=\frac{\beta^{q_2}}{|P(i\beta)|^2}\cdot \Re\left(i^{q_2}\overline{P(i\beta)}\right)
$$
where $P(s)=(q_1+q_2)s^{q_1+q_2-1}-a^\star_{11}q_2s^{q_2-1}-a^\star_{22}q_1s^{q_1-1}$. A simple computation leads to 
\begin{align*}\frac{\partial \Re(s)}{\partial a_{11}}\Big|_{(a^\star_{11},a^\star_{22})}&=
\delta^{\frac{q_2}{q_1+q_2}}\cdot\frac{\beta^{q_1+q_2-1}}{|P(i\beta)|^2}\cdot\sin\frac{(q_2-q_1)\pi}{2}\left(q_2\rho_1e^{q_2\omega}+q_1\rho_2e^{-q_1\omega}\right)= \\
&=\delta^{\frac{q_2}{q_1+q_2}}\cdot\frac{\beta^{q_1+q_2-1}}{|P(i\beta)|^2}\cdot\sin\frac{(q_2-q_1)\pi}{2}\cdot\frac{\partial h}{\partial\omega}(-\omega,q_1,q_2).
\end{align*}
In a similar way, we compute $\displaystyle\frac{\partial \Re(s)}{\partial a_{22}}\Big|_{(a^\star_{11},a^\star_{22})}$ and we finally obtain the gradient vector     
\begin{align*}\nabla &\Re(s)(a^\star_{11},a^\star_{22})=\left(\frac{\partial\Re(s)}{\partial a_{11}},\frac{\partial\Re(s)}{\partial a_{22}}\right)\bigg\rvert_{(a^\star_{11},a^\star_{22})}=\\&=\frac{\beta^{q_1+q_2-1}}{|P(i\beta)|^2}\cdot\sin\frac{(q_2-q_1)\pi}{2}\cdot\left(\delta^{\frac{q_2}{q_1+q_2}}\frac{\partial h}{\partial\omega}(-\omega,q_1,q_2),\delta^{\frac{q_1}{q_1+q_2}}\frac{\partial h}{\partial\omega}(\omega,q_1,q_2)\right).
\end{align*} 
From the parametric equations of the curve $\Gamma(\delta,q_1,q_2)$ and the properties of the function $h$ it is easy to deduce that 
the gradient vector $\nabla \Re(s)(a^\star_{11},a^\star_{22})$ is in fact a normal vector to the curve  $\Gamma(\delta,q_1,q_2)$ that points outward from the region below the curve. We deduce that the following transversality condition is fulfilled for the directional derivative:
$$\nabla_{\overline{u}} \Re(s)(a^\star_{11},a^\star_{22})=\left\langle\nabla \Re(z)(a^\star_{11},a^\star_{22}),\overline{u}\right\rangle>0,
$$
for any vector $\overline{u}$ which points outward from the region below the curve $\Gamma(\delta,q_1,q_2)$. Therefore, as the parameters $(a_{11},a_{22})$ cross the curve $\Gamma(\delta,q_1,q_2)$ into the region above the curve, $\Re(s)$ becomes positive and the pair of conjugated roots $(s,\overline{s})$ crosses the imaginary axis from the open left half-plane to the open right half-plane. Hence, $N(a_{11},a_{22},\delta,q_1,q_2)= 2$ for any $(a_{11},a_{22})$ from the region above the curve $\Gamma(\delta,q_1,q_2)$, and the system \eqref{sys.gen} is unstable.
\end{proof}

\begin{rem}
In Fig. \ref{fig.curves.frac}, several curves $\Gamma(\delta,q_1,q_2)$ have been plotted for $\delta=4$, $q_1=0.6$ and $q_2\in(0,1]$, together with the fractional-order-independent stability regions $R_s(\delta)$, $R_u(\delta)$. The regions below and above each curve represent the asymptotic stability region and instability region, respectively, provided by Theorem \ref{thm.q}. Lighter shades of red and blue have been used to plot the parts of these regions for which system \eqref{linearsys} is asymptotically stable / unstable for the particular values of the fractional orders $q_1,q_2$ which have been chosen, but not for all $(q_1,q_2)\in(0,1]$.
\end{rem}

\begin{figure}[htbp]
    \centering
    \includegraphics*[width=0.82\linewidth]{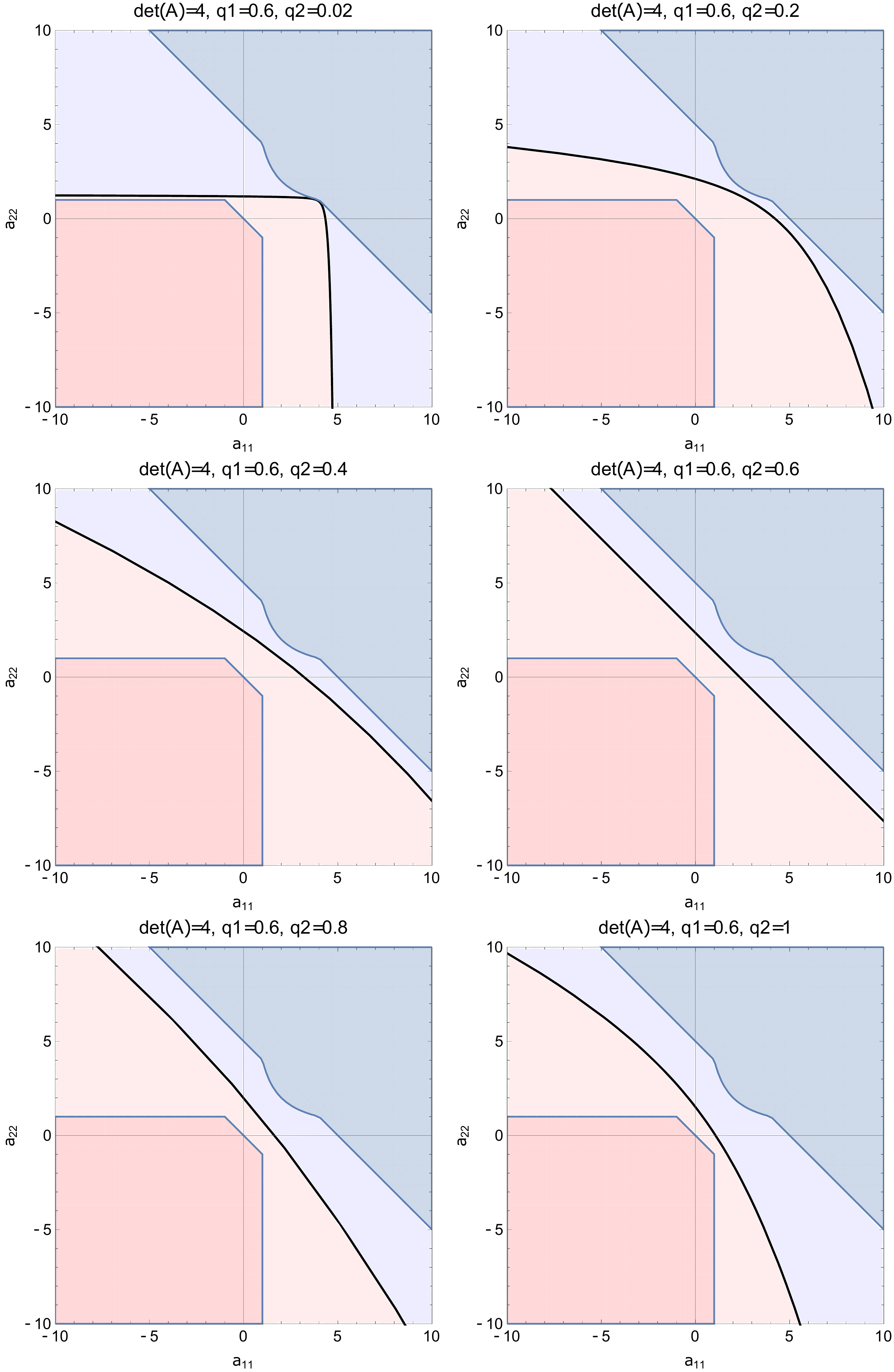}
    \caption{Curves $\Gamma(\delta,q_1,q_2)$ (black) for $\delta=4$, $q_1=0.6$ and  $q_2\in\{0.02,0.2,0.4,0.6,0.8,1\}$. Shades of red / blue represent the asymptotic stability / instability regions, lighter shades being associated to fractional-order-dependent regions and darker shades to fractional-order-independent regions $R_s(\delta)$ and $R_u(\delta)$. }
    \label{fig.curves.frac}
\end{figure}

Theorem \ref{thm.q} gives a relatively simple algebraic criterion (in the form of inequalities involving the elements of the system's matrix and the fractional orders) that permits to immediately decide the question of asymptotic stability or instability for a given two-dimensional system of fractional  differential equations.

\begin{ex}\label{ex}
We consider the system 
\begin{equation}\label{example}
\left\{\!
\begin{array}{l}
^c\!D^{q_1}x(t)=a_{11}x(t)+a_{12}y(t) \\
^c\!D^{q_2}y(t)=a_{21}x(t)+a_{22}y(t)
\end{array}
\right. \text{with }A=(a_{ij})=\begin{pmatrix}
    0.00001 & 1 \\
    -0.0022 & 0.1
\end{pmatrix}
\end{equation}
where $q_1,q_2\in(0,1]$. 

We first verify if the  fractional-order-independent stability or instability conditions given in Theorem \ref{thm.instab} and Theorem \ref{thm.stab} hold. First, as $\det(A)=0.002201>0$, a simple computation shows that $a_{11}+a_{22}<\det(A)+1$ and $a_{11}a_{22}<\det(A)$. Hence, based on Theorem \ref{thm.instab}, we deduce that system \eqref{example} is not unstable regardless of the considered fractional orders $q_1$ and $q_2$. On the other hand, as $a_{11}+a_{22}>0$, it is clear from Theorem \ref{thm.stab} that system \eqref{example} is not asymptotically stable regardless of the considered fractional orders $q_1$ and $q_2$. In other words, the stability and instability of system \eqref{example} depend on the choice of the fractional orders $q_1$ and $q_2$, as shown in the following situations.

\textbf{Case 1.} The special case $(q_1,q_2)=(\frac{1}{2},\frac{1}{4})$ has been considered in \cite{diethelm2017asymptotic}, and it has been shown, by transforming the corresponding system to a system of three fractional differential equations with the same order $\frac{1}{4}$, that in this particular case, system \eqref{example} is globally asymptotically stable. 

Indeed, applying Theorem \ref{thm.q}, we deduce that system \eqref{example} with the fractional orders $(q_1,q_2)=(\frac{1}{2},\frac{1}{4})$ is asymptotically stable if and only if 
$$a_{22}<\phi_{\delta,q_1,q_2}(a_{11}),$$
where $a_{11}=0.00001$, $a_{22}=0.1$, $\delta=\det(A)=0.002201$ and, relying on the notations introduced in Lemma \ref{lem.curve.gamma}:
\[\phi_{\delta,q_1,q_2}(a_{11})=\delta^{\frac{q_2}{q_1+q_2}}h(-\omega^*,q_1,q_2)\]
where $\omega^*$ is the unique root of the equation 
\[a_{11}=\delta^{\frac{q_1}{q_1+q_2}}h(\omega^*,q_1,q_2).\]
From this algebraic equation, we numerically compute $\omega^*=0.818108$ and therefore, we also deduce $\phi_{\delta,q_1,q_2}(a_{11})=0.208493$. As $a_{22}=0.1$, it can be easily seen that the asymptotic stability condition $a_{22}<\phi_{\delta,q_1,q_2}(a_{11})$ is satisfied (as depicted in Fig. \ref{fig:example}). 

\textbf{Case 2.} We now consider a different special case: $(q_1,q_2)=(\frac{1}{4},\frac{1}{2})$. Following the same steps as in the previous case, using Theorem \ref{thm.q}, we now compute: $\phi_{\delta,q_1,q_2}(a_{11})=0.0271274$ and therefore, as $a_{22}=0.1$, it follows that $a_{22}>\phi_{\delta,q_1,q_2}(a_{11})$, and hence, system \eqref{example} is unstable (as shown in Fig. \ref{fig:example}). 

This can also be verified by the method proposed in \cite{diethelm2017asymptotic}. Indeed, for $(q_1,q_2)=(\frac{1}{4},\frac{1}{2})$, system \eqref{example} is equivalent to the following system of three fractional differential equations of the same order $q=\frac{1}{4}$: 
\begin{equation}\label{ex.3d}
    ^c\!D^{q}\mathbf{z}(t)=B\mathbf{z}(t)\quad\text{with }B=\begin{pmatrix}
   a_{11} & a_{12} & 0 \\
   0 & 0 & 1 \\
    a_{21} & a_{22} & 0 
    \end{pmatrix}=
    \begin{pmatrix}
   0.00001 & 1 & 0 \\
   0 & 0 & 1 \\
 -0.0022 & 0.1 & 0 
    \end{pmatrix}
\end{equation}
The spectrum of eigenvalues of the matrix $B$ is \[\sigma(B)=\{-0.326701, 0.304593, 0.0221182\}\]
and hence, based on Matignon's theorem \cite{Matignon}, as $\arg(\lambda_2)=\arg(\lambda_3)=0$, it follows that system \eqref{ex.3d} is unstable.

In conclusion, it is easy to see from the previously considered cases that system \eqref{example} will be asymptotically stable for some pairs of fractional orders (such as $(q_1,q_2)=(\frac{1}{2},\frac{1}{4})$ considered in Case 1.), while it will be unstable for other pairs of fractional orders (such as $(q_1,q_2)=(\frac{1}{4},\frac{1}{2})$ considered in Case 2.). In fact, using the inequality provided by Theorem \ref{thm.q}, the region of all pairs of fractional orders $(q_1,q_2)\in(0,1]^2$ for which system \eqref{example} is asymptotically stable can be numerically computed (see Fig. \ref{fig.q1.q2}).  

\begin{figure}[htbp]
    \centering
    \includegraphics[width=0.9\linewidth]{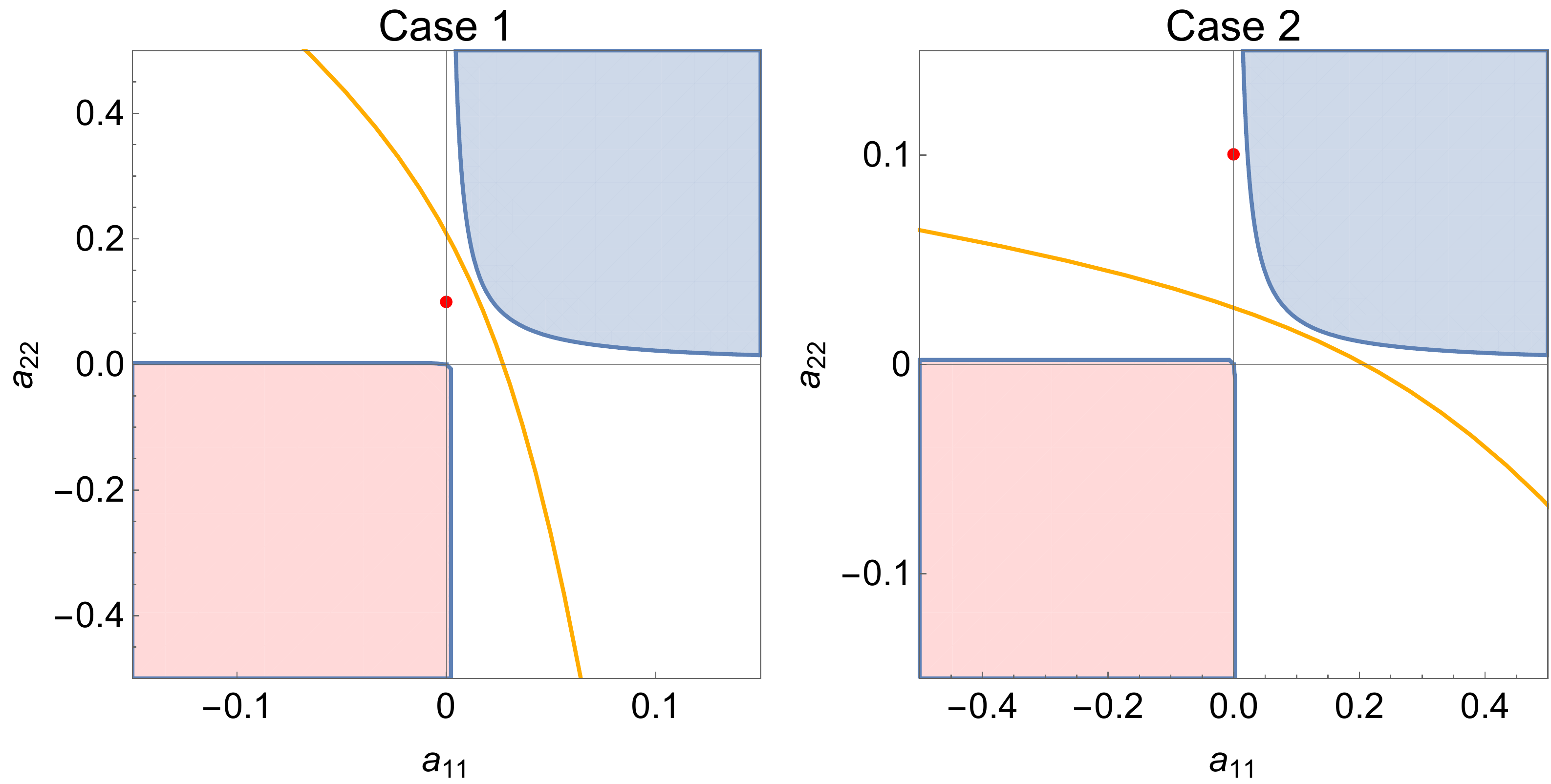}
    \caption{The position of the point $(a_{11},a_{22})=(0.00001,0.1)$ (plotted in red) with respect to the curve $\Gamma(\delta,q_1,q_2)$ (shown in orange) in the particular cases: Case~1:~ $(q_1,q_2)=(\frac{1}{2},\frac{1}{4})$ (left) and Case~2:~ $(q_1,q_2)=(\frac{1}{4},\frac{1}{2})$ (right) considered in Example \ref{ex}. }
    \label{fig:example}
\end{figure}

\begin{figure}[htbp]
    \centering
    \includegraphics[width=0.4\linewidth]{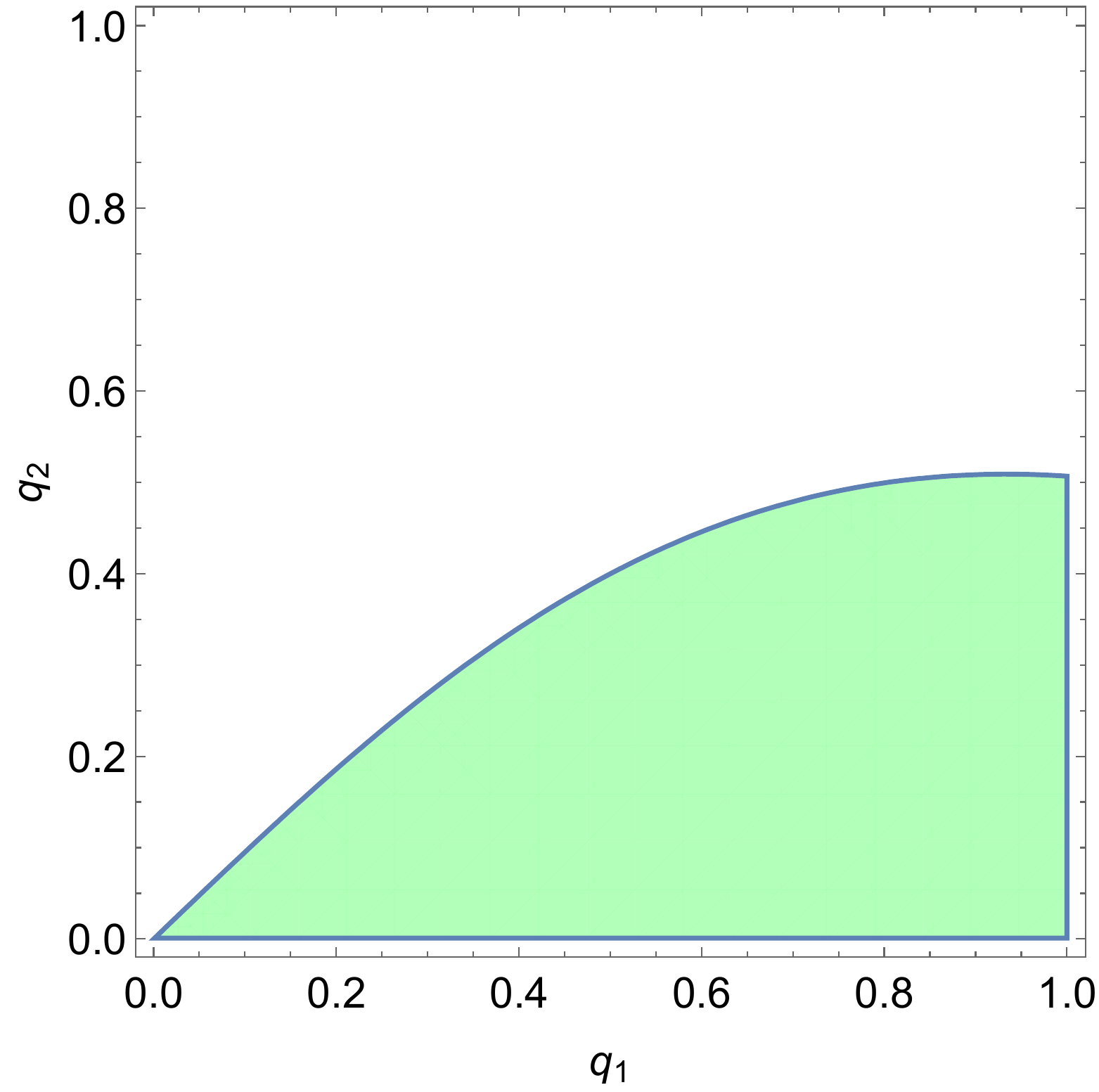}
    \caption{The region of fractional orders $(q_1,q_2)$ for which  system \eqref{example} is globally asymptotically stable.}
    \label{fig.q1.q2}
\end{figure}
\end{ex}

\subsection{Proofs of the fractional-order independent stability and instability results}\label{sec:proofs}
$ $

We are now ready to prove the main results presented in section \ref{sec:indep}. Throughout this section, we assume $\det(A)=\delta>0$, unless stated otherwise and we use the notations $R_u(\delta)$ and $R_s(\delta)$ introduced in section \ref{sec:indep} for the instability and stability regions, respectively. 

The following lemma provides a sufficient result for the instability of system \eqref{linearsys}, regardless of the fractional orders $q_1$ and $q_2$.
\begin{lemma} \label{lem.unstable}
If $(a_{11},a_{22})\in R_u(\delta)$, then system \eqref{linearsys} is unstable, regardless of the fractional orders $q_1$ and $q_2$.
\end{lemma}
\begin{proof}
Let $(a_{11},a_{22})\in R_u(\delta)$ and $(q_1,q_2)\in(0,1]^2$ arbitrarily fixed. We will show that the characteristic function $\Delta(s;a_{11},a_{22},\delta,q_1,q_2)$ has at least one positive real root. 

First, it is easy to see that $\Delta(s;a_{11},a_{22},\delta,q_1,q_2)\rightarrow\infty$ as $s\rightarrow\infty$.

On one hand, let us notice that if $a_{11}+a_{22}\geq \delta+1$, it follows that $$\Delta(1;a_{11},a_{22},\delta,q_1,q_2)=1-a_{11}-a_{22}+\delta\leq 0.$$ 
Hence, the function $s\mapsto \Delta(s;a_{11},a_{22},\delta,q_1,q_2)$ has at least one positive real root in the interval $[1,\infty)$. Therefore, the system \eqref{linearsys} is unstable. 

On the other hand, if $a_{11}>0$, $a_{22}>0$ and $a_{11}a_{22}\geq \delta$, as 
$$\Delta(s;a_{11},a_{22},\delta,q_1,q_2)=(s^{q_1}-a_{11})(s^{q_2}-a_{22})+\delta-a_{11}a_{22}$$
we see that for $s_0=(a_{11})^{1/q_1}>0$, we have $$\Delta(s_0;a_{11},a_{22},\delta,q_1,q_2)=\delta-a_{11}a_{22}\leq 0.$$ Hence, the function $s\mapsto \Delta(s;a_{11},a_{22},\delta,q_1,q_2)$ has at least one strictly positive real root. It follows that system \eqref{linearsys} is unstable. 
\end{proof}

The following lemma provides a sufficient result for the asymptotic stability of system \eqref{linearsys}, regardless of the fractional orders $q_1$ and $q_2$.
\begin{lemma}\label{lem.stable}
If $(a_{11},a_{22})\in R_s(\delta)$ then system \eqref{linearsys} is asymptotically stable, regardless of the fractional orders $q_1$ and $q_2$.
\end{lemma}

\begin{proof}
Let $(a_{11},a_{22})\in R_s(\delta)$ and $(q_1,q_2)\in(0,1]^2$ arbitrarily fixed. As $a_{11}+a_{22}<0$, we may assume, without loss of generality, that $a_{11}<0$. 

Assume by contradiction that $\Delta(s;a_{11},a_{22},\delta,q_1,q_2)$ has a root $s_0=re^{i\theta}$ in the right half-plane, where $r>0$ and $\theta\in\left[0,\frac{\pi}{2}\right]$. 

Multiplying the characteristic equation by $s_0^{-q_1}$, we get:
$$
s_0^{q_2}-a_{11}s_0^{q_2-q_1}-a_{22}+\delta s_0^{-q_1}=0.
$$
Taking the real part in the above equation and noticing that $s_0^{q_2}$, $s_0^{q_2-q_1}$ and $s_0^{-q_1}$ are in the right half-plane, we obtain:
\begin{align*}
a_{22}&=\Re(s_0^{q_2})-a_{11}\Re(s_0^{q_2-q_1})+\delta\Re(s_0^{-q_1})\\ &\geq \min\{1,-a_{11},\delta\}\left[\Re(s_0^{q_2})+\Re(s_0^{q_2-q_1})+\Re(s_0^{-q_1})\right]\\
&=\min\{1,-a_{11},\delta\}\left[r^{q_2}\cos (q_2\theta)+r^{q_2-q_1}\cos((q_2-q_1)\theta) +r^{-q_1}\cos(q_1\theta)\right].
\end{align*}
It is important to remark that for any $r>0$, $q_1,q_2\in(0,1]$ and $\theta\in\left[0,\frac{\pi}{2}\right]$, the following inequality holds: 
$$r^{q_2}\cos (q_2\theta)+r^{q_2-q_1}\cos((q_2-q_1)\theta) +r^{-q_1}\cos(q_1\theta)\geq 1.$$
Indeed, denoting $q_1\theta=x\in\left[0,\frac{\pi}{2} \right]$, $q_2\theta=y\in\left[ 0,\frac{\pi}{2} \right]$ and $r^{\frac{1}{\theta}}=\alpha>0$ this inequality is equivalent to
\begin{equation}\label{ineq}
\alpha^y\cos y+\alpha^{y-x}\cos(y-x)+\alpha^{-x}\cos x\geq 1,\quad \forall~x,y\in\left[0,\frac{\pi}{2} \right],~\alpha>0,
\end{equation}
which is proved in the Appendix A.2. It follows that:
$$a_{22}\geq \min\{1,-a_{11},\delta\}>0.$$
On the other hand, as $(a_{11},a_{22})\in R_s(\delta)$, we have $a_{22}<-a_{11}$ and $a_{22}<\min\{1,\delta\}$. Hence, $a_{22}<\min\{1,-a_{11},\delta\}$, which leads to a contradiction. Therefore, we deduce that all the roots of the characteristic function $\Delta(s;a_{11},a_{22},\delta,q_1,q_2)$ are in the open left half-plane, and hence, system  \eqref{linearsys} is asymptotically stable.
\end{proof}

As sufficiency in Theorems \ref{thm.instab} and \ref{thm.stab} has been proved in the previous two lemmas, the next part of this section is devoted to proving necessity in both theorems. With this aim in mind, in what follows, we will denote by $Q(\delta)$ the region of the $(a_{11},a_{22})$-plane which is covered by the curves $\Gamma(\delta,q_1,q_2)$ defined in Lemma \ref{lem.curve.gamma}, i.e.:
$$Q(\delta)=\{(a_{11},a_{22})\in\mathbb{R}^2~:~\exists (q_1,q_2)\in(0,1]^2~\text{s.t.}~(a_{11},a_{22})\in \Gamma(\delta,q_1,q_2)\}.$$
The following lemma is the key result which allows us to prove necessity in Theorems \ref{thm.instab} and \ref{thm.stab}.

\begin{lemma}\label{lem.Q}
The following holds:
$$Q(\delta)=\mathbb{R}^2\setminus\left(R_s(\delta)\cup R_u(\delta)\right).$$
\end{lemma}

\begin{proof} A proof by double inclusion is presented below.

\medskip
\noindent \emph{Step 1. Proof of the inclusion $Q(\delta)\subseteq\mathbb{R}^2\setminus\left(R_s(\delta)\cup R_u(\delta)\right)$.}

In the case $q_1=q_2=q$, elementary inequalities and Remark \ref{rem.q1.q2.q} provide that $\Gamma(\delta,q,q)$ are straight lines which are included in $\mathbb{R}^2\setminus\left(R_s(\delta)\cup R_u(\delta)\right)$. 

Let us now consider $q_1<q_2$ (the opposite case is treated similarly) and show that $\Gamma(\delta,q_1,q_2)\subset \mathbb{R}^2\setminus\left(R_s(\delta)\cup R_u(\delta)\right)$. Considering an arbitrary point $(a_{11},a_{22})\in \Gamma(\delta,q_1,q_2)$, it follows that there exists $\omega\in\mathbb{R}$ such that $a_{11}=\delta^{\frac{q_1}{q_1+q_2}}~h(\omega,q_1,q_2)$ and $a_{22}=\delta^{\frac{q_2}{q_1+q_2}}~h(-\omega,q_1,q_2)$, where the function $h$ is given in Lemma \ref{lem.curve.gamma}. 

Let us first show that $(a_{11},a_{22})\notin R_u(\delta)$. On one hand, one can write: 
$$a_{11}+a_{22}=u(t)+\delta u(t^{-1})$$
where $t=\delta^{\frac{1}{q_1+q_2}}e^\omega>0$ and $u(t)=\rho_2 t^{q_1}-\rho_1 t^{q_2}$, where the arguments of the functions $\rho_1$, $\rho_2$ have been dropped for simplicity. The function $u(t)$ reaches its maximal value at the point $t_{max}=\left(\frac{q_1\rho_2}{q_2\rho_1}\right)^{\frac{1}{q_2-q_1}}$ and a straightforward calculation leads to:
$$u_{max}=u(t_{max})=\left(\frac{\sin\frac{q_2\pi}{2}}{q_2}\right)^{\frac{q_2}{q_2-q_1}}\cdot\left(\frac{q_1}{\sin\frac{q_1\pi}{2}}\right)^{\frac{q_1}{q_2-q_1}}\cdot\frac{q_2-q_1}{\sin\frac{(q_2-q_1)\pi}{2}}.$$
We will next show that $u_{max}<1$. Indeed, as the function $v(x)=x\ln\left(\frac{x}{\sin x}\right)$ is positive and convex on $\left[0,\pi\right]$ with $\lim_{x\rightarrow 0}v(x)=0$, it follows that $v(x)$ is superadditive, and hence:
$$v(x)+v(y-x)< v(y)\quad,~\textrm{for any } 0<x<y\leq\frac{\pi}{2}.$$
Considering $x=\frac{q_1\pi}{2}$ and $y=\frac{q_2\pi}{2}$ in the previous inequality, we obtain 
$\ln(u_{max})<0$, and hence, $u_{max}<1$ for any $0<q_1<q_2\leq 1$. Therefore:
\begin{equation}\label{ineq.u1}
a_{11}+a_{22}=u(t)+\delta u(t^{-1})\leq u_{max}(\delta+1)<\delta+1.
\end{equation}

On the other hand, 
\begin{align*}
    a_{11}a_{22}&=\delta h(\omega,q_1,q_2)h(-\omega,q_1,q_2)=\\
    &=\delta(\rho_2e^{q_1\omega}-\rho_1e^{-q_2\omega})(\rho_2e^{-q_1\omega}-\rho_1e^{q_2\omega})=\\
    &=\delta \left[\rho_1^2+\rho_2^2-\rho_1\rho_2\left(e^{(q_1+q_2)\omega}+e^{-(q_1+q_2)\omega}\right)\right]<\\
    &< \delta \left(\rho_1^2+\rho_2^2-2\rho_1\rho_2\right)=\delta (\rho_2-\rho_1)^2=\\
    &=\delta \left(\frac{\cos\frac{(q_1+q_2)\pi}{4}}{\cos\frac{(q_2-q_1)\pi}{4}}\right)^2<\delta
\end{align*}
and hence, combined with inequality \eqref{ineq.u1} it follows that $(a_{11},a_{22})\notin R_u(\delta)$. 

Moreover, assuming by contradiction that $(a_{11},a_{22})\in R_s(\delta)$, Lemma \ref{lem.stable} implies that all roots of the characteristic function $\Delta(s;a_{11},a_{22},\delta,q_1,q_2)$ are in the open left half-plane, and hence, by Theorem \ref{thm.q} we obtain that $(a_{11},a_{22})\notin \Gamma(\delta,q_1,q_2)$, which is absurd. Therefore, $(a_{11},a_{22})\notin R_s(\delta)$. 

Hence, the proof of the inclusion $Q(\delta)\subseteq\mathbb{R}^2\setminus\left(R_s(\delta)\cup R_u(\delta)\right)$ is now complete.
\newpage
\medskip
\noindent \emph{Step 2. Proof of the inclusion $\mathbb{R}^2\setminus\left(R_s(\delta)\cup R_u(\delta)\right)\subseteq Q(\delta)$.}

Considering the function $F_\delta:\mathbb{R}\times(0,1]\times(0,1]\rightarrow\mathbb{R}^2$ defined by $$F_\delta(\omega,q_1,q_2)=\left(\delta^{\frac{q_1}{q_1+q_2}}~h(\omega,q_1,q_2),\delta^{\frac{q_2}{q_1+q_2}}~h(-\omega,q_1,q_2)\right),$$
it is easy to see that $Q(\delta)$ represents the image of the function $F_\delta$, i.e.: $Q(\delta)=F_\delta\left(\mathbb{R}\times(0,1]\times(0,1]\right)$.

From Remark \ref{rem.q1.q2.q}, it easily follows that
$$R_e(\delta)=\{(a_{11},a_{22})\in\mathbb{R}^2~:~0\leq a_{11}+a_{22}<2\sqrt{\delta}\}\subset Q(\delta).$$

Moreover, as $h(-\omega,q_1,q_2)=h(\omega,q_2,q_1)$ for any $q_1\neq q_2$, it follows that $Q(\delta)$ is symmetric with respect to the first bisector $a_{11}=a_{22}$ of the $(a_{11},a_{22})$-plane. Therefore, in order to determine $Q(\delta)$ it suffices to find its intersection with an arbitrary straight line $l_m:a_{11}-a_{22}=m$, $m\in\mathbb{R}$, which is parallel to the first bisector of the $(a_{11},a_{22})$-plane.  
First, Lemma \ref{lem.curve.gamma} implies that each curve $\Gamma(\delta,q_1,q_2)$ is the graph of a smooth, decreasing, concave, bijective function in the $(a_{11},a_{22})$-plane, and hence, it intersects the line $l_m$  exactly in one point. In other words, for arbitrarily fixed $q_1,q_2\in(0,1]$ and $m\in\mathbb{R}$, the equation
\begin{equation}\label{eq.om.star}\delta^{\frac{q_1}{q_1+q_2}}h(\omega,q_1,q_2)-\delta^{\frac{q_2}{q_1+q_2}}h(-\omega,q_1,q_2)=m
\end{equation}
has a unique solution $\omega_m^\star(q_1,q_2)$. From the implicit function theorem and the properties of the function $h$ it follows that the function $\omega^\star_m$ is continuously differentiable on the open sets 
\begin{align*}
S_{-}&=\{(q_1,q_2)\in(0,1)^2~:~q_1<q_2\}\\
S_{+}&=\{(q_1,q_2)\in(0,1)^2~:~q_1>q_2\}
\end{align*}
Therefore, the abscissa of the point of intersection $\Gamma(\delta,q_1,q_2)\cap l_m$ is 
$$a^m_{11}(q_1,q_2)=\delta^{\frac{q_1}{q_1+q_2}}h(\omega^\star_m(q_1,q_2),q_1,q_2).$$
The function $a^m_{11}$ is continuously differentiable on $S_-$ and $S_+$, and hence, $a^m_{11}(S_{\pm})$ are intervals. The problem of determining these intervals reduces to finding the extreme values of the function $a^m_{11}$ over the sets $S_-$ and $S_+$, respectively. 

Defining the functions $$\alpha_1(\omega,q_1,q_2)=\delta^{\frac{q_1}{q_1+q_2}}h(\omega,q_1,q_2)\quad\text{and}\quad  \alpha_2(\omega,q_1,q_2)=\delta^{\frac{q_2}{q_1+q_2}}h(-\omega,q_1,q_2),$$
from \eqref{eq.om.star} and the implicit function theorem it follows that 
$$\frac{\partial \omega^\star_m}{\partial q_k}(q_1,q_2)=-\left.\dfrac{\frac{\partial\alpha_1}{\partial q_k}-\frac{\partial\alpha_2}{\partial q_k}}{\frac{\partial\alpha_1}{\partial \omega}-\frac{\partial\alpha_2}{\partial \omega}}\right|_{(\omega_m^\star(q_1,q_2),q_1,q_2)}\quad,~ k=\overline{1,2}.$$

In what follows, we will show that the function $a_{11}^m$ does not have any critical points inside $S_\pm$. Indeed, assuming that $\nabla a_{11}^m(q_1,q_2)=0$ for $(q_1,q_2)\in S_\pm$, taking into account that $a_{11}^m(q_1,q_2)=\alpha_1(\omega_m^\star(q_1,q_2),q_1,q_2)$, a simple application of the chain rule leads to: 
$$\frac{\partial\alpha_1}{\partial\omega}(\omega_m^\star(q_1,q_2),q_1,q_2)\cdot \frac{\partial \omega^\star_m}{\partial q_k}(q_1,q_2)+\frac{\partial\alpha_1}{\partial q_k}(\omega_m^\star(q_1,q_2),q_1,q_2)=0\quad,~k=\overline{1,2}.$$
Combining the last two relations, it follows that: 
$$\frac{\partial\alpha_1}{\partial\omega}\cdot \frac{\partial\alpha_2}{\partial q_k}=\frac{\partial\alpha_1}{\partial q_k}\cdot \frac{\partial\alpha_2}{\partial \omega}\quad,~k=\overline{1,2},$$
where the arguments have been dropped for simplicity. Plugging in the expression of the function $h$ given in Lemma \ref{lem.curve.gamma} and eliminating $\delta$ from the previous system leads to a quadratic equation in $\xi=e^{(q_1+q_2)\omega^\star_m}$ which has a negative discriminant:  
$-\left(q_1^2\rho_2^2-q_2^2\rho_1^2\right)^2$, and hence, does not admit real roots.

Therefore, the extreme values of the function $a_{11}^m$ are reached on the boundaries of the sets $S_{\pm}$, respectively. This is equivalent to the fact that the boundary $\partial Q(\delta)$ is composed of points belonging to $\Gamma(\delta,q_1,q_2)$ when $(q_1,q_2)\in\partial S_{\pm}$. Hence, it remains to show that $\partial R_s(\delta)\cup\partial R_u(\delta)=\partial Q(\delta)$.

On one hand, due to the fact that $h(\omega,q_1,q_2)\rightarrow 1$ as $q_1\rightarrow 0$, for any $\omega\in\mathbb{R}$ and $q_2\in(0,1]$, it is easy to see that as $(q_1,q_2)\rightarrow (0,q)$, with $q\in(0,1]$, the curve $\Gamma(\delta,q_1,q_2)$ approaches the union of half-lines given parametrically by
$$H_1:~~\begin{cases}
a_{11}=1+\min\{0,t\}\\
a_{22}=\delta(1-\max\{0,t\})
\end{cases}\quad ,~t\in\mathbb{R}.$$

Similarly, due to the property $h(-\omega,q_1,q_2)=h(\omega,q_2,q_1)$ which holds for any $q_1\neq q_2$, we obtain that as $(q_1,q_2)\rightarrow (q,0)$, with $q\in(0,1]$, the curve $\Gamma(\delta,q_1,q_2)$ approaches the union of half-lines 
$$H_2:~~\begin{cases}
a_{11}=\delta(1+\min\{0,t\})\\
a_{22}=1-\max\{0,t\}
\end{cases}\quad ,~t\in\mathbb{R}.$$  

Moreover, Remark \ref{rem.q1.q2.q} provides that $$\Gamma(\delta,1,1):~~ a_{11}+a_{22}=0.$$

Therefore, a simple geometric analysis of the relative positions of the half-lines $H_1$ and $H_2$ given above and the line $a_{11}+a_{22}=0$ shows that $\partial R_s(\delta)\subset \partial Q(\delta)$. 

On the other hand, considering $\delta\neq 1$ and choosing $\omega=0$ in the parametric equations of the curve $\Gamma(\delta,q_1,q_2)$, $q_1\neq q_2$, given by Lemma \ref{lem.curve.gamma}, it follows that the points $$(a_{11}^0(q_1,q_2),a_{22}^0(q_1,q_2))=\left(\delta^{\frac{q_1}{q_1+q_2}}(\rho_2-\rho_1),\delta^{\frac{q_2}{q_1+q_2}}(\rho_2-\rho_1)\right)$$ belong to $Q(\delta)$. Let us also notice that the point $\left(\delta^{\frac{q_1}{q_1+q_2}},\delta^{\frac{q_2}{q_1+q_2}}\right)$ belongs to the arc of the parabola $P:~ a_{11}a_{22}=\delta$, considered between the points $(1,\delta)$ and $(\delta,1)$. Hence:
\begin{align*}
d((a_{11}^0,a_{22}^0),P)&\leq d\left((a_{11}^0,a_{22}^0),\left(\delta^{\frac{q_1}{q_1+q_2}},\delta^{\frac{q_2}{q_1+q_2}}\right)\right)=\\
&=|\rho_2-\rho_1-1|\sqrt{\delta^{\frac{2q_1}{q_1+q_2}}+\delta^{\frac{2q_2}{q_1+q_2}}}\leq \\
&\leq \sqrt{2}\max(1,\delta)|\rho_2-\rho_1-1|=\\
&=\sqrt{2}\max(1,\delta)\left(1-\frac{\cos\frac{(q_1+q_2)\pi}{4}}{\cos\frac{(q_2-q_1)\pi}{4}}\right)\longrightarrow 0,
\end{align*}
as either $q_1\rightarrow 0$ or $q_2\rightarrow 0$. Therefore, $P\subset \partial Q(\delta)$.

In a similar manner, considering $\omega=-\frac{\ln(\delta)}{q_1+q_2}$ in the parametric equations of the curve $\Gamma(\delta,q_1,q_2)$, $q_1\neq q_2$, given by Lemma \ref{lem.curve.gamma}, it follows that the points $(a_{11}^\delta(q_1,q_2),a_{22}^\delta(q_1,q_2))=(\rho_2-\delta\rho_1,\delta\rho_2-\rho_1)\in Q(\delta)$. For an arbitrary $\mu>0$, $\mu\neq 1$, let us consider the sequence of points $$M_n=\left(a_{11}^\delta\left(\frac{1}{n},\frac{\mu}{n}\right),a_{22}^\delta\left(\frac{1}{n},\frac{\mu}{n}\right)\right)\in Q(\delta), \quad n\in \mathbb{Z}^+,~n>[\mu].$$
Applying L'Hospital's rule results in
$$\lim_{n\rightarrow\infty} M_n=\left(\frac{\mu-\delta}{\mu-1},\frac{\delta\mu-1}{\mu-1}\right):=M_\mu.$$
It is now easy to deduce that the set of limit points $M_\mu$, with $\mu>0$, $\mu\neq 1$, is in fact the straight line $a_{11}+a_{22}=\delta+1$, except the segment joining the points of coordinates $(1,\delta)$ and $(\delta,1)$. Therefore, the straight line $a_{11}+a_{22}=\delta+1$ without the segment between $(1,\delta)$ and $(\delta,1)$ is also included in $\partial Q(\delta)$. Combined with the previous result concerning the arc of parabola $P$, it follows that $\partial R_u(\delta)\subset \partial Q(\delta)$.

The case $\delta=1$ is trivial, as the boundary $\partial R_u(\delta)$ becomes the whole straight line $a_{11}+a_{22}=2$, which is the limit of $\Gamma(1,q,q)$ as $q\rightarrow 0$. 

Hence, the proof is now complete.  
\end{proof}

\begin{rem}
In Fig. \ref{fig.regions.all}, for $\delta=4$, we exemplify the set $Q(\delta)$ and the results presented in Lemma \ref{lem.Q}, by plotting a large number of curves $\Gamma(\delta,q_1,q_2)$ for $(q_1,q_2)=\left(\frac{j}{40},\frac{k}{40}\right)$, with $j,k=\overline{1,40}$. The union of all these curves fills in the white region represented in Fig. \ref{fig.regiuni}, which separates the stability region $R_s(\delta)$ and the instability region $R_u(\delta)$.
\end{rem}

\begin{figure}[htbp]
		\centering
		\includegraphics*[width=0.5\linewidth]{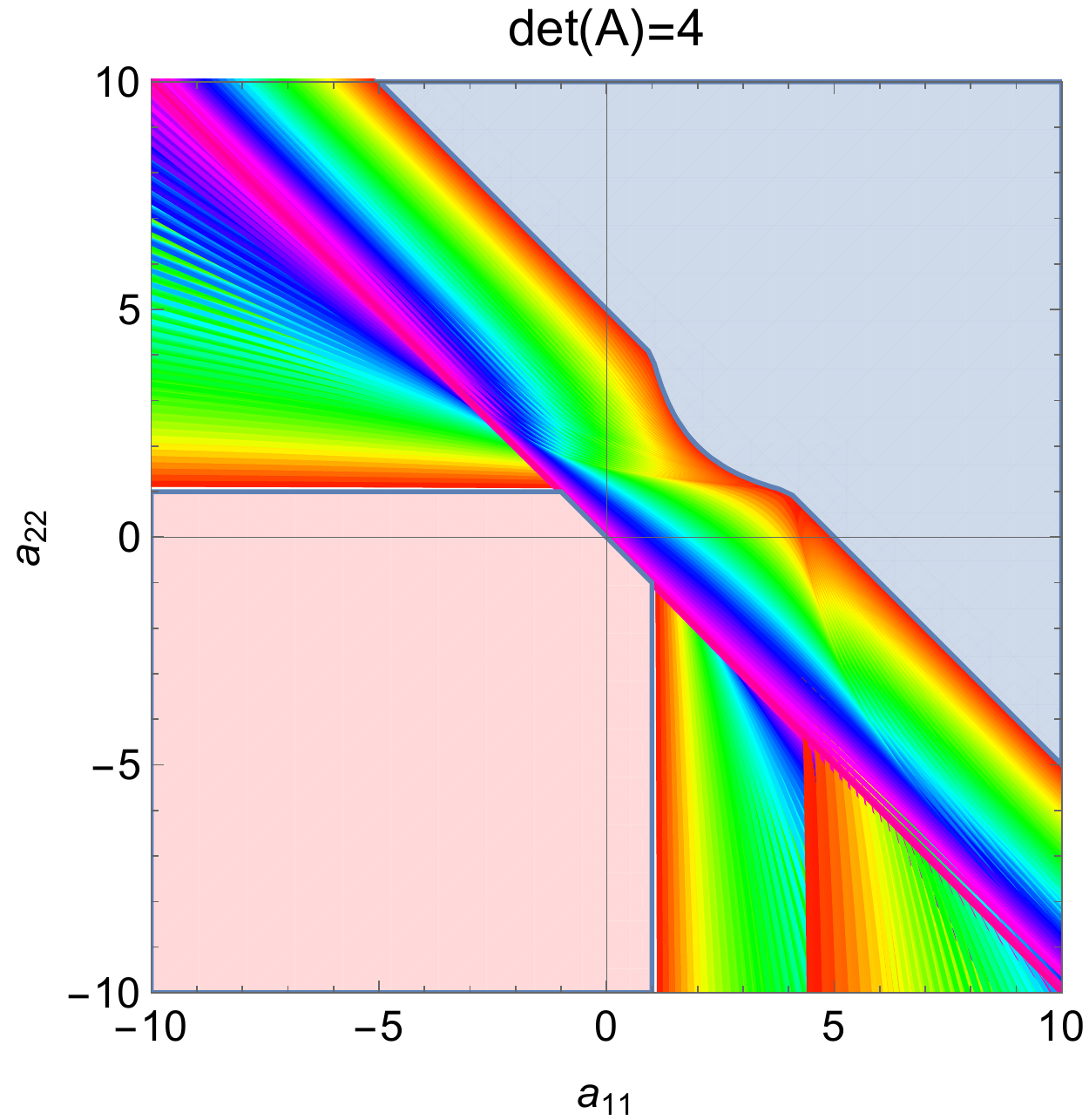}
			\caption{Curves $\Gamma(\delta,q_1,q_2)$ given by Lemma \ref{lem.curve.gamma}, for $\det(A)=\delta=4$ and  $q_i\in\left\{\frac{k}{40},~k=\overline{1,40}\right\}$, $i=\overline{1,2}$ ($1600$ curves), color-coded from red to violet according to increasing values of $q_1q_2$. The red/blue shaded regions represent the sets $R_u(\delta)$ and $R_s(\delta)$, respectively.}
	\label{fig.regions.all}
\end{figure}

We finally present the proofs of the main theorems.
\newpage
\begin{proof}[Proof of Theorem \ref{thm.instab}]$ $

\emph{Proof of statement (i).}
Because $\Delta(0)=\delta<0$ and $\Delta(\infty)=\infty$, due to the fact that $\Delta$ is continuous on $(0,\infty)$, it results that it has at least one strictly positive real root. Therefore, based on Proposition \ref{thm.lin.stab}, it follows that system \eqref{linearsys} is unstable.

\emph{Proof of statement (ii).}
If $\Delta(0)=\delta>0$, sufficiency is provided by Lemma \ref{lem.unstable}. For the proof of necessity, assuming that system \eqref{linearsys} is unstable, regardless of the fractional orders $q_1$ and $q_2$, and assuming by contradiction that $(a_{11},a_{22})\notin R_u(\delta)$, using Lemma \ref{lem.Q} it follows that there exist $q_1^*,q_2^*\in (0,1]$ (not unique) such that $(a_{11},a_{22})$ is in the connected component of $\mathbb{R}^2\setminus \Gamma(\delta,q_1^*,q_2^*)$ which includes $R_s(\delta)$, i.e. $(a_{11},a_{22})$ is below the curve $\Gamma(\delta,q_1^*,q_2^*)$. Hence, based on Theorem \ref{thm.q}, it follows that system \eqref{linearsys} with the particular fractional orders $q_1^*,q_2^*$ is asymptotically stable, which is absurd. \end{proof}

\begin{proof}[Proof of Theorem \ref{thm.stab}]$ $
Sufficiency is provided by Lemma \ref{lem.stable}. As for the proof of necessity, let us assume that system \eqref{linearsys} is asymptotically stable, regardless of the fractional orders $q_1$ and $q_2$, and assume by contradiction that $(a_{11},a_{22})\notin R_a(\delta)$. Lemma \ref{lem.Q} provides that there exist $q_1^*,q_2^*\in (0,1]$ (not unique) such that $(a_{11},a_{22})$ is in the connected component of $\mathbb{R}^2\setminus \Gamma(\delta,q_1^*,q_2^*)$ which includes $R_u(\delta)$, i.e. $(a_{11},a_{22})$ is above the curve $\Gamma(\delta,q_1^*,q_2^*)$. Hence, based on Theorem \ref{thm.q}, it follows that system \eqref{linearsys} with the particular fractional orders $q_1^*,q_2^*$ is not asymptotically stable, which contradicts the initial hypothesis.
\end{proof}

\section{Conclusions}

In this work, a complete characterization of fractional-order-independent stability and instability properties of two-dimensional incommensurate linear fractional-order systems has been achieved. Moreover, necessary and sufficient conditions have also been presented for the stability and instability of two-dimensional fractional-order systems, depending on the choice of the fractional orders of the Caputo derivatives. These results provide comprehensive practical tools for a straightforward stability analysis of two-dimensional fractional-order systems encountered in real world applications.  

Extension of these results to the case of two-dimensional systems of fractional-order difference equations requires further investigation. A possible generalization to higher-dimensional fractional-order systems is still an open question which will be addressed in future research, taking into account the increasing complexity of the problem.

\section*{Appendix} 

\subsection*{A.1. Boundedness of the set of unstable roots of $\Delta(s;a_{11},a_{22},\delta,q_1,q_2)$}

 The characteristic equation of system \eqref{linearsys} is
 $$s^{q_1+q_2}-a_{11}s^{q_2}-a_{22}s^{q_1}+\delta=0.$$
Denoting $\alpha=\dfrac{q_1+q_2}{2}$ and $\beta=\dfrac{q_2-q_1}{2}$, with $0\leq \beta<\alpha\leq 1$, the characteristic equation can be written as
$$s^{2\alpha}-a_{11}s^{\alpha+\beta}-a_{22}s^{\alpha-\beta}+\delta=0.$$
Dividing by $\sqrt{\delta}s^\alpha$, we obtain:
\begin{equation}\label{eq.alpha.beta}\dfrac{s^\alpha}{\sqrt{\delta}}+\dfrac{\sqrt{\delta}}{s^\alpha}=\dfrac{a_{11}}{\sqrt{\delta}}s^\beta+\dfrac{a_{22}}{\sqrt{\delta}}s^{-\beta}.
\end{equation}
Denoting $z=\dfrac{s^\alpha}{\sqrt{\delta}}$, $c_1=a_{11}(\sqrt{\delta})^{\frac{\beta}{\alpha}-1}$, $c_2=a_{22}(\sqrt{\delta})^{-\frac{\beta}{\alpha}-1}$ and $c=(c_1,c_2)$, equation \eqref{eq.alpha.beta} becomes
$$z+z^{-1}=c_1z^{\frac{\beta}{\alpha}}+c_2z^{-\frac{\beta}{\alpha}}.$$
Denoting $p=\dfrac{q_1+q_2}{2\min\{q_1,q_2\}}\geq 1$ and $q=\dfrac{\alpha}{|\beta|}=\dfrac{q_1+q_2}{|q_1-q_2|}\geq 1$ it follows that $\dfrac{1}{p}+\dfrac{1}{q}=1$, and Young's inequality provides:
\begin{align}
\nonumber \left| |z|-|z|^{-1} \right|&\leq |z+z^{-1}|\leq |c_1|\cdot |z|^{\frac{\beta}{\alpha}}+|c_2|\cdot |z|^{-\frac{\beta}{\alpha}}&\\
\nonumber &\leq \dfrac{1}{p}(|c_1|^p+|c_2|^p)+\dfrac{1}{q}\left( |z|+|z|^{-1} \right)&\\
 \label{ineq.z}   &= \left( 1-\dfrac{1}{q} \right)\|c\|_p^p+\dfrac{1}{q}\left( |z|+|z|^{-1} \right).
\end{align}
On one hand, if $|z|>|z|^{-1}$, or equivalently $|z|>1$, inequality \eqref{ineq.z} can be written as the quadratic inequality
$$ |z|^2-\|c\|_p^p\cdot |z|- \gamma\leq 0,$$
where $\gamma=\dfrac{q+1}{q-1}$, which in turn, implies that
\begin{equation}\label{ineq.z1} |z|\leq \|c\|_p^p+\sqrt{\gamma}.
\end{equation}
On the other hand, if $|z|<|z|^{-1}$, or equivalently  $|z|<1$, inequality \eqref{ineq.z} leads to the quadratic inequality
$$\gamma|z|^2+\|c\|_p^p\cdot |z|-1\geq 0,$$
and hence:
\begin{equation}
\label{ineq.z2}|z|\geq \dfrac{-\|c\|_p^p+\sqrt{\|c\|_p^{2p}+4\gamma}}{2\gamma}
\end{equation}
In the above calculations, $$\|c\|_p^p=|c_1|^p+|c_2|^p=|a_{11}|^p(\sqrt{\delta})^{p\left( \frac{\beta}{\alpha}-1 \right)}+|a_{22}|^p(\sqrt{\delta})^{-p\left( \frac{\beta}{\alpha}+1 \right)}.$$
Furthermore, as
$$p\left( \dfrac{\beta}{\alpha}-1 \right)=\dfrac{q_1+q_2}{2\min\{q_1,q_2\}}\cdot \left( \dfrac{q_2-q_1}{q_2+q_1}-1 \right)=\dfrac{-q_1}{\min\{q_1,q_2\}}$$
$$-p\left( \dfrac{\beta}{\alpha}+1 \right)=-\dfrac{q_1+q_2}{2\min\{q_1,q_2\}}\cdot \left( \dfrac{q_2-q_1}{q_2+q_1}+1 \right)=\dfrac{-q_2}{\min\{q_1,q_2\}}.$$
we have:
$$
    \|c\|_p^p=|a_{11}|^p(\sqrt{\delta})^{-\frac{q_1}{\min\{q_1,q_2\}}}+|a_{22}|^p(\sqrt{\delta})^{-\frac{q_2}{\min\{q_1,q_2\}}}\leq D(\delta,q_1,q_2)\cdot \|a\|_p^p,
$$
where $D(\delta,q_1,q_2)=\max\left \{ (\sqrt{\delta})^{-\frac{q_1}{\min\{q_1,q_2\}}},(\sqrt{\delta})^{-\frac{q_2}{\min\{q_1,q_2\}}} \right \}$.

Therefore, inequalities \eqref{ineq.z1} and \eqref{ineq.z2} provide that 
\begin{equation}\label{ineq.z.final}
    \dfrac{-\|c\|_p^p+\sqrt{\|c\|_p^{2p}+4\gamma}}{2\gamma}\leq |z|\leq \|c\|_p^p+\sqrt{\gamma}
\end{equation}
Considering the decreasing function $f:\mathbb{R}^+\to\mathbb{R}^+$ defined by $$f(u)=\dfrac{-u+\sqrt{u^{2}+4\gamma}}{2\gamma}$$ and the increasing function $F:\mathbb{R}^+\to\mathbb{R}^+$ $$F(u)=u+\sqrt{\gamma},$$ inequality \eqref{ineq.z} becomes $$f(\|c\|_p^p)\leq |z|\leq F(\|c\|_p^p).$$

Taking into consideration that $\|c\|_p^p\leq D(\delta,q_1,q_2)\|a\|_p^p$ and $z=\dfrac{s^\alpha}{\sqrt{\delta}}$, the previous inequality implies
$$
    f\left(D(\delta,q_1,q_2)\|a\|_p^p\right)\leq f(\|c\|_p^p)\leq \dfrac{|s|^\alpha}{\sqrt{\delta}}\leq F(\|c\|_p^p)\leq F\left(D(\delta,q_1,q_2)\|a\|_p^p\right),$$
and thus:
$$\left( \sqrt{\delta}f(D\left(\delta,q_1,q_2)\|a\|_p^p\right) \right)^{\frac{1}{\alpha}}\leq |s|\leq \left( \sqrt{\delta}F\left(D(\delta,q_1,q_2)\|a\|_p^p\right) \right)^{\frac{1}{\alpha}}.
$$

Therefore, considering the decreasing function $l_{\delta,q_1,q_2}:\mathbb{R}^+\to\mathbb{R}^+$ defined by $$ l_{\delta,q_1,q_2}(v)=\left( \sqrt{\delta}f\left( D(\delta,q_1,q_2)v^p \right) \right)^{\frac{1}{\alpha}} $$ and the increasing function $L_{\delta,q_1,q_2}:\mathbb{R}^+\to\mathbb{R}^+$ defined by $$ L_{\delta,q_1,q_2}(v)=\left( \sqrt{\delta}F\left( D(\delta,q_1,q_2)v^p \right) \right)^{\frac{1}{\alpha}} $$ inequality \eqref{ineq.s} is obtained.

\subsection*{A.2. Proof of inequality \eqref{ineq}.}

Because of symmetry, it suffices to prove inequality \eqref{ineq} for $\alpha\geq 1$, i.e.
$$
\alpha^y\cos y+\alpha^{y-x}\cos(y-x)+\alpha^{-x}\cos x\geq 1,\quad \forall~x,y\in\left[0,\frac{\pi}{2} \right],~\alpha\geq 1.
$$
Denoting $h(x)=\alpha^{y-x}\cos(y-x)+\alpha^{-x}\cos x$, its derivative is
$$h'(x)=-\alpha^{y-x}\ln(\alpha)\cos(y-x)+\alpha^{y-x}\sin (y-x)-\alpha^{-x}\ln(\alpha)\cos x-\alpha^{-x}\sin x$$
The equation $h'(x)=0$ is equivalent to
$$\tan x=\frac{\alpha^{y}\sin y-\alpha^{y}\ln(\alpha)\cos y-\ln(\alpha)}{1+\alpha^{y}\ln(\alpha)\sin y+\alpha^{y}\cos y}$$
which has a solution $x^*(y)$ on the interval $\left[0,\frac{\pi}{2}\right)$ if and only if the numerator of the right-hand term of the above equations positive, i.e. 
\begin{equation}\label{ineq.appendix.a2.1}\alpha^{y}(\sin y-\ln(\alpha)\cos y)\geq \ln(\alpha).
\end{equation}

If inequality \eqref{ineq.appendix.a2.1} does not hold, it means in fact that $h'(0)<0$, which implies $h'(x)<0$, for any $x\in\left(0,\frac{\pi}{2}\right)$. Therefore the  function $h$ is decreasing and its minimal value is  $h\left(\frac{\pi}{2}\right)=\alpha^{y-\frac{\pi}{2}}\sin y$.

Otherwise, if inequality \eqref{ineq.appendix.a2.1} holds, i.e. $h'(0)\geq 0$, it turns out that $x^*(y)$ is a maximum point of $h(x)$ and the function $h$ is increasing on the interval $(0,x^*(y))$ and decreasing on the interval $\left(x^*(y),\frac{\pi}{2}\right)$. Therefore, the minimal value of the function $h$ is either  $h(0)=\alpha^y\cos y+1$ or $h\left(\frac{\pi}{2}\right)=\alpha^{y-\frac{\pi}{2}}\sin y$. However, it is easy to see that $\alpha^{y-\frac{\pi}{2}}\sin y\leq 1$, for any $y\in\left( 0,\frac{\pi}{2}\right)$, and hence, the minimal value of the function $h$ is   $h\left(\frac{\pi}{2}\right)=\alpha^{y-\frac{\pi}{2}}\sin y$.

Therefore, we obtain that 
$$h(x)\geq \alpha^{y-\frac{\pi}{2}}\sin y, \quad \forall~x,y\in \left[0,\frac{\pi}{2}\right],~\alpha\geq 1,$$
which leads to:
\begin{equation}\label{ineq.appendix.a2.2}\alpha^y\cos y+\alpha^{y-x}\cos(y-x)+\alpha^{-x}\cos x \geq \alpha^y\cos y+\alpha^{y-\frac{\pi}{2}}\sin y.\end{equation}
Considering the function $g(y)=\alpha^y\cos y+\alpha^{y-\frac{\pi}{2}}\sin y$ and its derivative
$$g'(y)=\alpha^y\ln(\alpha)\cos y-\alpha^y \sin y+\alpha^{y-\frac{\pi}{2}}\ln(\alpha)\sin y+\alpha^{y-\frac{\pi}{2}}\cos y,$$
we obtain that $g'(y)=0$ if and only if
$$y=y^*=\arctan\left(\frac{\ln(\alpha)+\alpha^{-\frac{\pi}{2}}}{1-\ln(\alpha)\alpha^{-\frac{\pi}{2}}}\right).$$
It can be easily seen that $y^*$ is a local maximum point for the function $g$ on the interval $\left(0,\frac{\pi}{2}\right)$, and hence, the minimal values of $g$ are reached in $g(0)=g\left(\frac{\pi}{2}\right)=1$. Therefore, $g(y)\geq 1$, for any $y\in \left[0,\frac{\pi}{2}\right]$, and combined with  \eqref{ineq.appendix.a2.2}, we obtain inequality (\ref{ineq}).

\bibliographystyle{plain}

\bibliography{bibliografie}
\end{document}